\documentclass[11pt,a4paper]{article}
\usepackage{amsmath}
\usepackage{amssymb}
\usepackage{amsthm}
\usepackage{amsfonts}
\usepackage{enumerate}
\usepackage{pstricks}
\usepackage{pst-plot}
\usepackage{pst-poly}
\usepackage{algorithm}
\usepackage{algorithmic}
\usepackage{caption,subcaption} 
\usepackage{mathtools} 
\usepackage{url}
\usepackage{hyperref}
\usepackage{breakurl,url}

\newcommand{\mace}[1]{{{#1}}}
\newcommand{\mna}[1]{{\mathcal{#1}}}

\newcommand{\omace}[1]{\mbox{$\overline{\mace{#1}}$}} 
\newcommand{\umace}[1]{\mbox{$\underline{\mace{#1}}$}} 

\def\Mid#1{#1_c}
\def\Rad#1{#1_\Delta}

\newcommand{\R}[0]{{\mathbb{R}}}

\newcommand{\N}[0]{{\mathbb{N}}}

\def\eps{{\varepsilon}}
\newcommand{\mmid}[0]{;\,}		
\newcommand{\Ss}[0]{\mbox{\large$\Sigma$}}	
\newcommand{\nic}[1]{{}}

\newcommand{\seznam}[1]{{\{1, \ldots, {#1}\}}}

\def\clqq{``}
\def\crqq{''}
\def\quo#1{\clqq{}#1\crqq{}}  

\newcommand{\st}[0]{{\ \ \mbox{subject to}\ \ }}
\DeclareMathOperator{\sgn}{sgn}	

\def\nref#1{\eqref{#1}}

\newtheorem{theorem}{Theorem}

\newtheorem{proposition}{Proposition}
\newtheorem{lemma}{Lemma}
\newtheorem{corollary}{Corollary}

\theoremstyle{definition}

\begin{document}

\title{Properties of the solution set of absolute value equations and the related matrix classes}

\author{Milan Hlad\'{i}k\thanks{Charles University, Faculty  of  Mathematics  and  Physics,
Department of Applied Mathematics, 
Malostransk\'e n\'am.~25, 11800, Prague, Czech Republic,
  (\href{mailto:hladik@kam.mff.cuni.cz}{hladik@kam.mff.cuni.cz}, \url{https://kam.mff.cuni.cz/\string~hladik}).
  }}

\maketitle

\begin{abstract}
The absolute value equations (AVE) problem is an algebraic problem of solving $Ax+|x|=b$. So far, most of the research focused on methods for solving AVEs, but we address the problem itself by analysing properties of AVE and the corresponding solution set. In particular, we investigate topological properties of the solution set, such as convexity, boundedness, connectedness, or whether it consists of finitely many solutions. Further, we address problems related to nonnegativity of solutions such as solvability or unique solvability. AVE can be formulated by means of different optimization problems, and in this regard we are interested in how the solutions of AVE are related with optima, Karush--Kuhn--Tucker points and feasible solutions of these optimization problems. 

We characterize the matrix classes associated with the above mentioned properties and inspect the computational complexity of the recognition problem; some of the classes are polynomially recognizable, but some others are proved to be NP-hard. For the intractable cases, we propose various sufficient conditions. We also post new challenging problems that raised during the investigation of the problem. 
\end{abstract}

\textbf{Keywords:}\textit{ absolute value equations, linear complementarity problem, special matrices, Interval analysis, NP-hardness.}

\textbf{AMS:}\textit{ 65G40, 90C33, 15Bxx.}

\section{Introduction}

The absolute value equations (AVE) problem reads
\begin{align*}
Ax+|x|=b,
\end{align*}
where $A\in\R^{n\times n}$ and $b\in\R^n$. 
The demonstration of the equivalence of AVE and the linear complementarity problem \cite{Man2007,ManMey2006}, as well as other applications, stimulated intensive research in AVE; see \cite{Man2015,Pro2009,WuLi2018,ZamHla2021a,ZamHla2020ab} among many others. 
The problem is still challenging since many issues related to AVE and computationally hard, e.g., checking solvability of AVE~\cite{Man2007}, or checking uniqueness~\cite{Pro2009}. 

So far, most of the research has been directed towards developing methods for solving AVEs. Our focus is different. Our aim is to inspect the properties of AVE and the solution set. In particular, we want to find out under which conditions the solution set is nonempty, convex, bounded, connected, or it consists of finitely many solutions. We also address the questions regarding nonnegativity of solutions. 

All these properties are reflected in the properties of matrix~$A$. Similarly as for the linear complementarity problem \cite{Cot2010,CotPan2009,Mur1997}, this naturally defines the matrix classes associated with the particular properties of AVE. We characterize these matrix classes and show which are easy to recognize and, in contrast, which are intractable.

\paragraph{Notation}
$D_v$ is the diagonal matrix with entries given by vector~$v$, $I_n$ is the identity matrix of size $n\times n$, and $e_i$ is the $i$th canonical unit vector (i.e., the $i$th column of $I_n$). The all-ones vector of appropriate length is denoted by~$e$, and $v_I$ stands for the subvector of $v$ defined by the index set~$I$. Given a matrix $A\in\R^{n\times n}$, $\rho(A)$ stands for its spectral radius, $\sigma_{\min}(A)$ for the minimal singular value, $\|A\|_2$ for the spectral norm, $A_{i*}$ is the $i$th row of $A$, and $A_{i_1:i_2,j_1:j_2}$ for the submatrix $(a_{ij})$, where $i_1\leq i\leq i_2$ and $j_1\leq j\leq j_2$. 
The sign of a real $r$ is $\sgn(r)=1$ if $r>0$, $\sgn(r)=-1$ if $r<0$ and $\sgn(r)=0$ if $r=0$.

An interval matrix is defined as a set of matrices $[\umace{A},\omace{A}]=\{A'\in\R^{m\times n}\mmid \umace{A}\leq A'\leq \omace{A}\}$, where the inequality is understood entrywise. Often, we simply write the interval matrix as $[\Mid{A}\pm\Rad{A}]$, where $\Mid{A}=\frac{1}{2}(\umace{A}+\omace{A})$ is the midpoint matrix and $\Rad{A}=\frac{1}{2}(\omace{A}-\umace{A})$ is the radius matrix. 
An interval matrix is called \emph{regular} if every $A\in [\umace{A},\omace{A}]$ is nonsingular; otherwise it is called \emph{irregular}.

\section{The solution set and solvability}\label{sSolv}

\paragraph{The solution set}
The solution set of AVE is naturally defined as
\begin{align*}
\Ss=\Ss(b)=\{x\in\R^n\mmid Ax+|x|=b\};
\end{align*}
we will use the notation $\Ss(b)$ if the vector $b$ is subject to changes.

The basic properties of the solution set follow from the orthant decomposition of the space $\R^n$. Recall that the orthant defined by the sign vector $s\in\{\pm1\}^n$ is characterized by $D_sx\geq0$. Within this orthant, we can replace $|x|$ by $D_sx$, so that the solution set is a convex polyhedron there.

\begin{proposition}\label{propDecomp}
The solution set $\Ss$ located in the orthant given by $s\in\{\pm1\}^n$ is characterized by the linear system $(A+D_s)x=b$, $D_sx\geq0$.
\end{proposition}

As a consequence, if finite, then $\Ss$ consists of at most $2^n$ points. 
Figure~\ref{figSolSet} illustrates various shapes of the solution set~$\Ss$.

\begin{figure}[t]
\begin{subfigure}[b]{0.47\textwidth}
\centering
\psset{unit=8ex,arrowscale=1.5}
\begin{pspicture}(-1.3,-1.3)(2.1,2.1)
\psaxes[ticksize=2pt,labels=all,ticks=all, Oy=0, Dx=1, Dy=1, showorigin=false]{->}(0,0)(-1.2,-1.2)(2,2)
\uput[-135](-0.04,-0.04){$0$}
\psline[linewidth=1.5pt,linecolor=blue](1,0)(0,1)
\uput[45](0.51,0.51){\blue$\Ss$}
\end{pspicture}
\caption{
$\begin{pmatrix}1&2\\2&1\end{pmatrix}x+|x|=\begin{pmatrix}2\\2\end{pmatrix}$.\label{figSolSet1}}
\end{subfigure}
\hfill
\begin{subfigure}[b]{0.47\textwidth}
\centering
\psset{unit=8ex,arrowscale=1.5}
\begin{pspicture}(-1.3,-1.5)(2.6,1.6)
\psaxes[ticksize=2pt,labels=all,ticks=all, Oy=0, Dx=1, Dy=1, showorigin=false]{->}(0,0)(-1.2,-1.2)(2.5,1.5)
\uput[-135](-0.05,-0.08){$0$}
\psline[linewidth=1.5pt,linecolor=blue](-1.2,-1.2)(0,0)(2.5,0)
\uput[90](1.5,0.1){\blue$\Ss$}
\end{pspicture}
\caption{
$\begin{pmatrix}-1&2\\0&1\end{pmatrix}x+|x|=\begin{pmatrix}0\\0\end{pmatrix}$.\label{figSolSet2}}
\end{subfigure}

\begin{subfigure}[b]{0.47\textwidth}
\centering
\psset{unit=8ex,arrowscale=1.5}
\begin{pspicture}(-1.3,-1.3)(2.6,2.8)
\psaxes[ticksize=2pt,labels=all,ticks=all, Oy=0, Dx=1, Dy=1, showorigin=false]{->}(0,0)(-1.2,-1.2)(2.5,2.5)
\uput[-135](-0.04,-0.04){$0$}
\psline[linewidth=1.5pt,linecolor=blue](0,1)(1.4,2.4)
\uput[-45](0.7,1.7){\blue$\Ss$}
\end{pspicture}
\caption{
$\begin{pmatrix}1&-2\\2&-3\end{pmatrix}x+|x|=\begin{pmatrix}-2\\-2\end{pmatrix}$.\label{figSolSet3}}
\end{subfigure}
\hfill
\begin{subfigure}[b]{0.47\textwidth}
\centering
\psset{unit=8ex,arrowscale=1.5}
\begin{pspicture}(-1.5,-1.6)(2.6,2.3)
\psaxes[ticksize=2pt,labels=all,ticks=all, Oy=0, Dx=1, Dy=1, showorigin=false]{->}(0,0)(-1.4,-1.4)(2.5,2)
\uput[-135](-0.05,-0.065){$0$}
\psline[linewidth=1.5pt,linecolor=blue](1,0)(2.4,1.4)
\uput[45](0.2,0.2){\blue$\Ss$}
\psset{linecolor=blue}
\qdisk(-1,-1){2pt}
\end{pspicture}
\caption{
$\begin{pmatrix}1&-2\\2&-3\end{pmatrix}x+|x|=\begin{pmatrix}2\\2\end{pmatrix}$.\label{figSolSet4}}
\end{subfigure}

\begin{subfigure}[b]{0.47\textwidth}
\centering
\psset{unit=8ex,arrowscale=1.5}
\begin{pspicture}(-2.3,-1.9)(1.6,1.6)
\psaxes[ticksize=2pt,labels=all,ticks=all, Oy=0, Dx=1, Dy=1, showorigin=false]{->}(0,0)(-2.2,-1.6)(1.5,1.4)
\uput[-135](-0.05,-0.065){$0$}
\psline[linewidth=1.5pt,linecolor=blue](-2,-1)(0,-1)(1,0)
\uput[-45](0.5,-0.5){\blue$\Ss$}
\end{pspicture}
\caption{
$\begin{pmatrix}1&-2\\0&1\end{pmatrix}x+|x|=\begin{pmatrix}2\\0\end{pmatrix}$.\label{figSolSet5}}
\end{subfigure}
\hfill
\begin{subfigure}[b]{0.47\textwidth}
\centering
\psset{unit=8ex,arrowscale=1.5}
\begin{pspicture}(-2.2,-1.6)(2.2,2.3)
\psaxes[ticksize=2pt,labels=all,ticks=all, Oy=0, Dx=1, Dy=1, showorigin=false]{->}(0,0)(-1.9,-1.4)(1.9,2)
\uput[-135](-0.045,-0.099){$0$}
\psline[linewidth=1.5pt,linecolor=blue](-1.8182,-1.1237)(0,0)(1.1111,1.7978)
\uput[-40](1,1.6){\blue$\Ss$}
\end{pspicture}
\caption{
$\begin{pmatrix}\sqrt{5}&-2\\2&-\sqrt{5}\end{pmatrix}x+|x|=\begin{pmatrix}0\\0\end{pmatrix}$.\label{figSolSet6}}
\end{subfigure}

\caption{Six examples of various shapes of the solution set of AVE.\label{figSolSet}}
\end{figure}

\paragraph{Solvability} 
First we recall the result on unique solvability for an arbitrary right-hand side. Notice that it is open whether AVEs with this property are efficiently solvable~\cite{GharGil2012}.

\begin{theorem}[\cite{WuLi2018}]
The AVE has a unique solution for each $b\in\R^n$ if and only if $[A\pm I_n]$ is regular.
\end{theorem}

Checking this property is NP-hard since AVE is equivalent to the linear complementarity problem and the unique solvability therein is equivalent to the P-matrix property; see \cite{Cox1994,Roh2012c}.

Regularity of $[A\pm I_n]$ is tractable in some cases. The generally known tractable interval matrix classes are M-matrices, strongly regular matrices and some others; see the overview in~\cite{Hla2020a}. In particular, any of the following two conditions is sufficient to ensure regularity of $[A\pm I_n]$:
\begin{align}
\label{condRegSuff1}
\rho(|A^{-1}|)&<1,\\
\label{condRegSuff2}
\sigma_{\min}(A)&>1.
\end{align}
In both cases, the unique solution can be computed in polynomial time. For the former see \cite{ZamHla2021a}, and for the latter see~\cite{ManMey2006}.

We can show tractability for symmetric interval matrices, too. Recall that a signature of a symmetric matrix is the triple of the numbers of its positive, negative and zero eigenvalues.
We denote by $\lambda_i(B)$ the $i$th largest eigenvalue of a symmetric matrix~$B$.

\begin{proposition}
If $A\in\R^{n\times n}$ is symmetric, then $[A\pm I_n]$ is regular if and only if both matrices $A-I_n$ and $A+I_n$ have the same signature.
\end{proposition}

\begin{proof}
By \cite{HlaDan2010}, each eigenvalue of $B\in[A\pm I_n]$ lies in the interval bounds $\lambda_i(B)\in[\lambda_i(A)-1,\lambda_i(A)+1]$. Since every value in $[\lambda_i(A)-1,\lambda_i(A)+1]$ is obtained for a certain matrix $B\in[A\pm I_n]$ (in particular, the value in $\lambda_i(A)-\alpha$ is attained for $A-\alpha I_n$), we have that interval $[\lambda_i(A)-1,\lambda_i(A)+1]$ is the exact range of $i$th eigenvalues of matrices in $[A\pm I_n]$. Thus $[A\pm I_n]$ is regular if and only if no such interval contains zero. 
\end{proof}

\section{Nonnegativity}

Some issues concerning nonnegativity of the solutions are tractable. Nevertheless, nonnegativity also raises some new open questions.

\begin{proposition}\label{propSolvNonneg}
The AVE has a unique nonnegative solution for each $b\geq0$ if and only if $(A+I_n)^{-1}\geq0$.
\end{proposition}

\begin{proof}
\quo{If.} 
By Proposition~\ref{propDecomp}, nonnegative solutions are characterized by the system 
$(A+I_n)x=b$, $x\geq0$. From this equation we derive $x=(A+I_n)^{-1}b\geq0$, so there is a nonnegative solution and is unique.

\quo{Only if.} 
If $(A+I_n)$ is singular, then there cannot be a unique nonnegative solution for each~$b$. Thus it is nonsingular and suppose to the contrary that $(A+I_n)^{-1}_{ij}<0$ for some $i,j$. Put $b\coloneqq e_j$. Then the vector $x\coloneqq (A+I_n)^{-1}b$ satisfies $x_i<0$, so there is no nonnegative solution; a contradiction.
\end{proof}

\begin{proposition}
The AVE has a nonnegative solution for each $b\geq0$ if and only if it has a unique nonnegative solution for each $b\geq0$.
\end{proposition}

\begin{proof}
The \quo{If} part is obvious, so we focus on the \quo{Only if} part. By the assumption, for each $b\geq0$, the system
\begin{align*}
(A+I_n)x=b,\ x\geq0
\end{align*}
is feasible. By the Farkas lemma~\cite{Schr1998}, there is no $y$ such that
\begin{align*}
(A+I_n)^Ty\geq0,\ b^Ty<0.
\end{align*}
That is, for every $b\geq0$ and for every $y$ such that $(A+I_n)^Ty\geq0$ we have $b^Ty\geq0$. Therefore the convex polyhedral cone described by $(A+I_n)^Ty\geq0$ lies in the nonnegative orthant. This cannot happen when $A+I_n$ is singular (in which case the cone contains a whole line), so it is nonsingular. 
We have in particular that all edges of the cone are nonnegative. Since these edges are the columns of $(A+I_n)^{-T}$, we get $(A+I_n)^{-T}\geq0$. Eventually, apply Proposition~\ref{propSolvNonneg}.
\end{proof}

A natural question appears here whether unique solvability of AVE for each $b\geq0$ (or $b\leq0$) is sufficient for unique solvability of AVE for every~$b$; it remains open. 

Proposition~\ref{propSolvNonneg} can easily be extended to other orthants, too. In total, it yields the following sufficient condition for solvability.

\begin{proposition}\label{propSolSufOrth}
The AVE has a solution for each $b\geq0$ if $(AD_s+I_n)^{-1}\geq0$ for some $s\in\{\pm1\}^n$. 
In which case, there is a unique solution in the orthant $D_sx\geq0$. 
\end{proposition}

\begin{proof}
Suppose that $(AD_s+I_n)^{-1}\geq0$ for some $s\in\{\pm1\}^n$. By Proposition~\ref{propDecomp}, solutions lying in the orthant $D_sx\geq0$ are characterized by the system 
$(A+D_s)x=b$, $D_sx\geq0$. Since $(AD_s+I_n)^{-1}\geq0$ and $b\geq0$, we have that the vector defined by $x^b\coloneqq D_s(AD_s+I_n)^{-1}b$ lies in the orthant $D_sx\geq0$. We also verify that it solves the AVE since 
$Ax^b+|x^b|
=(A+D_s)x^b
=(A+D_s)D_s(AD_s+I_n)^{-1}b=b$.
\end{proof}

Checking the condition from Proposition~\ref{propSolSufOrth} might not be easy; we suspect that the problem is NP-hard. It becomes tractable in certain cases, e.g., when $A$ has rank one; in this can we simply use the Sherman--Morrison formula.

Anyway, a more efficiently computable condition is of interest. Below, we state one such condition. It is more general than that from \cite[Prop.~5]{ManMey2006}, which employs a matrix norm instead of the spectral radius. 

\begin{corollary}\label{corSolNonnegRho}
Let $\rho(A)<1$ and $A\leq0$. Then the AVE has a unique nonnegative solution for each $b\geq0$.
\end{corollary}

\begin{proof}
From the asumptions, $A+I_n$ is an M-matrix, so Proposition~\ref{propSolvNonneg} applies.
\end{proof}

Now, we present an extension of Corollary~\ref{corSolNonnegRho} to cover all orthants. Notice that checking the assumptions of Proposition~\ref{propSolvNonnegRhoOrth} is an easy task.

\begin{proposition}\label{propSolvNonnegRhoOrth}
Let $\rho(A)<1$ and $AD_s\leq0$ for some $s\in\{\pm1\}^n$. Then the AVE has a solution for each $b\geq0$. In which case, there is a unique solution in the orthant $D_sx\geq0$. 
\end{proposition}

\begin{proof}
In the orthant $D_sx\geq0$, the solutions are characterized by the system 
$(A+D_s)x=b$, $D_sx\geq0$. Substitute $y\equiv D_sx$ to obtain $(AD_s+I_n)y=b$, $y\geq0$. 
Now, apply Corollary~\ref{corSolNonnegRho}. 
\end{proof}

\paragraph{Inverse nonnegative matrices} 
Recall that an interval matrix $[\umace{A},\omace{A}]$ is \emph{inverse nonnegative} if $A^{-1}\geq0$ for each $A\in[\umace{A},\omace{A}]$. By the theorem of Kuttler \cite{Kut1971}, an interval matrix $[\umace{A},\omace{A}]$ is inverse nonnegative if and only if $\umace{A}^{-1}\geq0$ and $\omace{A}^{-1}\geq0$. We will also need the following lemma.

\begin{lemma}\label{lmmKuttlerExt}
If $[\umace{A},\omace{A}]$ is regular and $\omace{A}^{-1}\geq0$, then $[\umace{A},\omace{A}]$ is inverse nonnegative.
\end{lemma}

\begin{proof}
Since $[\umace{A},\omace{A}]$ is regular, $\omace{A}^{-1}\geq0$ and the derivatives of $\omace{A}^{-1}$ with respect to the matrix entries are nonpositive~\cite{PetPed2012}, we have that $A^{-1}\geq0$ for every $A\in[\umace{A},\omace{A}]$.
\end{proof}

Since inverse nonnegative interval matrices are regular, the corresponding AVE has a unique solution for each $b\in\R^n$. If $b\geq0$, then we have an explicit formula for it, $x=(A+I_n)^{-1}b\geq0$. Thus: 

\begin{proposition}\label{propSolvInvNonnegNonneg}
If $[A\pm I_n]$ is inverse nonnegative, then for each $b\geq0$ the AVE has a unique solution and this solution is nonnegative.
\end{proposition}

In the general case, the solution is efficiently computable by linear programming \cite{ZamHla2021a}. However, we can say more -- the problem is strongly polynomial. Our algorithm is inspired by the method for finding the interval hull of interval systems with inverse nonnegative matrices~\cite{Neu1990}.

\begin{algorithm}[t]
\caption{Solving AVE with $[A\pm I_n]$ inverse nonnegative\label{algInvNonneg}}
\begin{algorithmic}[1]
\STATE
put $x^0\coloneqq (A+I_n)^{-1}b$, $k\coloneqq 0$
\REPEAT
\STATE
$s^k\coloneqq\sgn(x^k)$
\STATE
$x^{k+1}\coloneqq(A+D_{s^k})^{-1}b$
\STATE
$k\coloneqq k+1$
\UNTIL{$x^k=x^{k-1}$}
\RETURN
$x^k$
\end{algorithmic}
\end{algorithm}

\begin{proposition}
If $[A\pm I_n]$ is inverse nonnegative, then Algorithm~\ref{algInvNonneg} returns the unique solution of AVE in at most $n$ iterations.
\end{proposition}

\begin{proof}
For any $k>0$ we have 
$$
(A+D_{s^k})x^k
=Ax^k+|x^k|
\geq Ax^k+D_{s^{k-1}}x^k
=(A+D_{s^{k-1}})x^k=b.
$$
By inverse nonnegativity, we have
$$
x^k
\geq (A+D_{s^k})^{-1}b
=x^{k+1}.
$$
Hence the sequence of vectors $x^0,x^1,x^2,\dots$ forms a (componentwise) nonincreasing sequence. This proves that the number of iterations is finite. The number of $n$ for the maximum number of iterations follows from the fact that the sign vector $s^k$ changes only if some $x_i^k$ changes its sign (from positive value to negative), which happens at most $n$-times.
\end{proof}

Proposition~\ref{propSolvInvNonnegNonneg} can be generalized as follows.

\begin{proposition}\label{propSolvInvNonnegOrth}
If $[AD_s\pm I_n]$ is inverse nonnegative for some $s\in\{\pm1\}^n$, then for each $b\geq0$ the AVE has a unique solution and this solution lies in the orthant $D_sx\geq0$.
\end{proposition}

\begin{proof}
Unique solvability follows from regularity of $[AD_s\pm I_n]$. In the orthant $D_sx\geq0$, the solution are characterized by the system $(A+D_s)x=b$, $D_sx\geq0$. From the equation we derive $D_sx=D_s(A+D_s)^{-1}b=(AD_s+I_n)^{-1}b\geq0$, so the solution fits the right orthant.
\end{proof}

It is not hard to verify the assumptions of Proposition~\ref{propSolvInvNonnegOrth}. One just calculates $A^{-1}$, which provides a sign pattern. If $\sgn(A^{-1})=se^T$ for some $s\in\{\pm1\}^n$, then the problem reduces to checking inverse nonnegativity of $[AD_s\pm I_n]$; otherwise, the assumption cannot be satisfied.

It is an open question whether the converse implication in Proposition~\ref{propSolvInvNonnegNonneg} is valid. We can state it under stronger assumptions on regularity of $[A\pm I_n]$.

\begin{proposition}
Let $[A\pm I_n]$ be regular. If for each $b\geq0$ the AVE has a unique solution and this solution is nonnegative, then $[A\pm I_n]$ is inverse nonnegative.
\end{proposition}

\begin{proof}
From Proposition~\ref{propSolvNonneg} we get $(A+I_n)^{-1}\geq0$. 
From Lemma~\ref{lmmKuttlerExt}, we have the rest. 
\end{proof}

Now, we can give a complete characterization of inverse nonnegative interval matrices in terms of solvability of AVE.

\begin{proposition}
The interval matrix $[A\pm I_n]$ is inverse nonnegative if and only if for each $b\geq0$ the AVE has a unique nonnegative solution and for each $b\leq0$ the AVE has a unique nonpositive solution.
\end{proposition}

\begin{proof}
\quo{Only if.} Obvious in view of Proposition~\ref{propSolvInvNonnegNonneg}.

\quo{If.}
By Proposition~\ref{propSolvNonneg}, if AVE has a unique nonnegative solution for each $b\geq0$, then $(A+I_n)^{-1}\geq0$. Analogously we can show that if AVE has a unique nonpositive solution for each $b\leq0$, then $(A-I_n)^{-1}\geq0$. By Kuttler's theorem, $[\umace{A},\omace{A}]$ is inverse nonnegative.
\end{proof}

\section{Finite and infinite number of solutions}

\paragraph{The infinite case} 
We first discuss the question of how many orthants may possess infinitely many solutions and under which assumptions it may happen.

\begin{proposition}
There is no AVE such that each orthant contains infinitely many solutions.
\end{proposition}

\begin{proof}
Suppose to the contrary that such AVE system exists. Then matrix $A+D_s$ is singular for every $s\in\{\pm1\}^n$. By \cite{Hla2018e} this means that every matrix in $[A\pm I_n]$ is singular. By \cite{Hla2018e} again we have that $I_n$ is singular as well, a contradiction.
\end{proof}

Even though there cannot be infinitely many solutions in all orthants, there can be in all orthants but one. Consider, for example, the system $x+|x|=0$. All orthants contain infinitely many solutions, except for the positive orthant.

The above example works, but, as we will see in Proposition~\ref{propInfOrthBezJed}, it is in some sense degenerate. 
Consider now the AVE system
\begin{align*}
\begin{pmatrix}0_{n-1} & -e\\0 & 1\end{pmatrix}x+|x|=
\begin{pmatrix}2e\\0 \end{pmatrix}.
\end{align*}
All $2^{n-1}$ orthants with the negative last coordinate contain infinitely many solutions in their interiors. In particular, the orthant given by $s\in\{\pm1\}^{n-1}\times\{-1\}$ contains the interior point solution $x=s$. 

\begin{proposition}\label{propInfOrthBezJed}
If the AVE has infinitely many solutions in $2^n-1$ orthants, then $b=0$, $\det(A)=a_{11}\cdots a_{nn}$ and $|a_{ii}|=1$ for every $i=1,\dots,n$.
\end{proposition}

\begin{proof}
From the assumption, $A+D_s$ is singular for each $s\in\{\pm1\}^n$ but possibly one $s^*$. Let $i\in\seznam{n}$ be arbitrary and fix $s_i\coloneqq -s^*_i$. Then for every admissible sign vector $s$, the matrix $A+D_s$ is singular. By \cite{CohDan1996,CohJoh1989}, one of the submatrices $(A+D_s)_{1:i,i:n}$ or $(A+D_s)_{i:n,1:i}$ must be zero. This shows that $\det(A)=a_{11}\cdots a_{nn}$ and $|a_{ii}|=1$ for every $i=1,\dots,n$.

We claim that there is $k$ such that $|A_{k*}|=e_k^T$. We can see it by induction. From the above observation with $i\coloneqq n$, we have $|A_{n*}|=e_n^T$ or $|A_{*n}|=e_n$. In the first case, we are done. In the second case, we inductively inspect the upper-left submatrix. 

In fact, the statements can be generalized: for every $\ell$, there exist $\ell$ rows of $A$ such that each of them has at most $\ell$ nonzero elements (one of them is on the diagonal).

Let $k$ be such that $|A_{k*}|=e_k^T$. Then $b_k=0$ since otherwise the $k$-th equation in the AVE is unsolvable either for $x_k\geq0$, or for $x_k\leq0$. As a consequence, we necessarily have $x_k=0$ when restricted to $x_k\geq0$ or to $x_k\leq0$.
Now, take $k'$ such that the $k'$ row of $A$ has at most two nonzeros $a_{k',k'}$ and $a_{k',k}$. From the above reasoning, we have $x_k=0$ in half of the orthants, so we again deduce $b_{k'}=0$. 
By induction, we proceed further until we show $b=0$.
\end{proof}

\begin{proposition}
Let $n=2$ and $b\not=0$. Then the solution set of AVE cannot have infinitely many solutions in the opposite quadrants.
\end{proposition}

\begin{proof}
Suppose to the contrary that there is an AVE system $Ax+|x|=b$ such that there are infinitely many solutions in the orthants represented by $\pm D$, where $|D|=I_n$. Since the system $(A+D)x=b$ has infinitely many solutions, there is $y\not=0$ such that $y^T(A+D)=0^T$ and $y^Tb=0$. Similarly, there is $y'\not=0$ such that $y'^T(A-D)=0^T$ and $y'^Tb=0$. Since $n=2$ and $b\not=0$, both vectors $y$ and $y'$ represent the same direction. Thus we can assume without loss of generality that $y'=y$. However, this implies $y^T(A+D)=0^T=y^T(A-D)$, whence $y^TD=0^T$. Therefore $y=0$; a contradiction. 
\end{proof}

If $b=0$, then the situation may occur. Figure~\ref{figSolSet2} shows the counterexample, and Figure~\ref{figSolSet6} presents a system with infinitely many solutions in the interiors of the opposite quadrants.
If $n>2$, then the situation may occur, too. Consider the AVE system with
$$
A=\begin{pmatrix}6 & -8 & 3\\ 9 & -13 & 5\\ 12 & -18 & 7\end{pmatrix},\quad
b=\begin{pmatrix}1\\ 3\\ 6\end{pmatrix}.
$$
The solution set in the nonnegative orthant forms a line passing through point $x=(1,3,6)^T$, and similarly in the nonpositive orthant it forms a line passing through point $-x$.

\paragraph{The finite case} 
Now, we characterize the case where the solution set $\Ss(b)$ is finite (possibly empty) for every right-hand side vector~$b$. The property turns out to be hard to check, so we also present several sufficient conditions.

\begin{proposition}\label{propFinChar}
The set $\Ss(b)$ is finite for each $b\in\R^n$ if and only if $A+D_s$ is nonsingular for each $s\in\{\pm1\}^n$.
\end{proposition}

\begin{proof}
\quo{If.} 
By Proposition~\ref{propDecomp}, the solution set lying in the orthant given by $s\in\{\pm1\}^n$ is described by $(A+D_s)x=b$, $D_sx\geq0$. By nonsingularity of $A+D_s$, there is none or one solution located in this orthant.

\quo{Only if.} 
Suppose to the contrary that $A+D_s$ is singular for certain $s\in\{\pm1\}^n$. Define $b\coloneqq (A+D_s)s$. In the orthant given by $s$, there lies the solution $x=s$. Since $A+D_s$ is singular and $x=s$ lies in the interior of the orthant, there are infinitely many solutions as well.
\end{proof}

Deciding on the above property is intractable even for rank-one matrices.

\begin{proposition}
Checking whether $A+D_s$ is nonsingular for each $s\in\{\pm1\}^n$ is co-NP-hard an a class of problems with $A$ having rank one and natural entries.
\end{proposition}

\begin{proof}
Consider the matrix in the form $A=ev^T$ for some $v\in\N^n$. Now,  $A+D_s$ is nonsingular if and only if $D_SA+I_n=sv^T+I_n$ is nonsingular. By the Sherman--Morrison formula, we equivalently have $v^Ts\not=-1$. This is a variant of the NP-hard subset sum problem, in which natural numbers $v_1,\dots,v_n$ are given, and one asks to split them into two subsets the sums of which differ by one.
\end{proof}

Obviously, unique solvability of AVE for each $b\in\R^{n}$, i.e., regularity of $[A\pm I_n]$, implies a finite number of solutions. When matrix $A$ is small enough, $[A\pm I_n]$ is often not regular and we cannot use this condition. However, for small matrices we have the following criterion.
We use the relation $A\succ B$ to denote that $A-B$ is positive definite.

\begin{corollary}
$\Ss(b)$ is finite for each $b\in\R^n$ if and only if $A^TA+I_n+D_sA+A^TD_s\succ0$ for each $s\in\{\pm1\}^n$.
\end{corollary}

\begin{proof}
Matrix $A+D_s$ is nonsingular if and only if matrix  $(A+D_s)^T(A+D_s)=A^TA+I_n+D_sA+A^TD_s$ is positive definite.
\end{proof}

There are two practical conditions as a result.

\begin{corollary}
$\Ss(b)$ is finite for each $b\in\R^n$ if $\rho(|A|+|A|^T)<1+\lambda_{\min}(A^TA)$.
\end{corollary}

\begin{proof}
We have $\rho(-D_sA-A^TD_s)\leq\rho(|A|+|A|^T)<1+\lambda_{\min}(A^TA)$ for each $s\in\{\pm1\}^n$. Consequently, $-D_sA-A^TD_s\prec I_n+\lambda_{\min}(A^TA)I_n\preceq I_n+A^TA$.
\end{proof}

\begin{corollary}
$\Ss(b)$ is finite for each $b\in\R^n$ if $\|A\|_2<\frac{1}{2}$.
\end{corollary}

\begin{proof}
If $\|A\|_2<\frac{1}{2}$, then 
$$
\rho(D_sA+A^TD_s)
\leq \|D_sA+A^TD_s\|_2
\leq \|D_sA\|_2+\|A^TD_s\|_2
=2\|A\|_2<1.
$$
Thus $A^TA+I_n+D_sA+A^TD_s\succeq I_n+D_sA+A^TD_s\succ 0$. 
\end{proof}

We present one more sufficient condition, which is incomparable with the above ones.

\begin{proposition}\label{propBoundSuff}
The set $\Ss(b)$ is finite for each $b\in\R^n$ if $\rho(|A|)<1$.
\end{proposition}

\begin{proof}
In view of Proposition~\ref{propFinChar}, suppose to the contrary that $x\not=0$ solves $Ax+D_sx=0$ for some $s\in\{\pm1\}^n$. Then $|x|=|-D_sx|=|Ax|\leq|A|\cdot|x|$. Define $B\coloneqq |A|+\eps ee^T>0$ for $\eps>0$ sufficiently small. Then $|x|\leq B |x|$, whence $y\leq By$ for $y\coloneqq B\cdot|x|>0$. By the Perron--Frobenius theory \cite{HorJoh1985,Mey2000} we have $\rho(B)\geq1$. Therefore $\rho(A)\geq1$ as well; a contradiction.
\end{proof}

We now extend the condition such that $C=I_n$ reduces the above case.

\begin{proposition}\label{propBoundSuffC}
The $\Ss(b)$ is finite for each $b\in\R^n$ if $\rho(|CA|+|I_n-C|)<1$ for certain $C\in\R^{n\times n}$.
\end{proposition}

\begin{proof}
Suppose to the contrary that $x\not=0$ solves $Ax+D_sx=0$ for some $s\in\{\pm1\}^n$. Then it solves $CAx+CD_sx=0$, from which $D_sx=-CAx+(I_n-C)D_sx$. We derive $|x|=|D_sx|\leq(|CA|+|I_n-C|)|x|$, and the rest is analogous to the proof of Proposition~\ref{propBoundSuff}.
\end{proof}

Surprisingly, it turns out that Proposition~\ref{propBoundSuffC} is not more general than Proposition~\ref{propBoundSuff}; both conditions hold simultaneously. The proof is inspired by the technique from~\cite{RexRoh1995}, but we cannot use it directly since matrix $A$ can be singular.

\begin{proposition}
If $\rho(|CA|+|I_n-C|)<1$ for certain $C\in\R^{n\times n}$, then $\rho(|A|)\leq \rho(|CA|+|I_n-C|)$.
\end{proposition}

\begin{proof}
First we show that $C$ must be nonsingular, which follows from
\begin{align}\label{ineqPfPropBoundSuffEq}
\rho(I_n-C)
\leq \rho(|I_n-C|)
\leq \rho(|CA|+|I_n-C|)<1.
\end{align}
Now, define
$$
G\coloneqq  |CA|+|I_n-C|+\eps ee^T,\quad
\alpha\coloneqq  \rho(G)<1,
$$
where $\eps>0$ is small enough.
Since $G>0$, by the Perron--­Frobenius theorem there is $x>0$ such that $Gx=\alpha x$. 
Using also $\alpha<1$ we derive
$$
|CA|x+\alpha |I_n-C|x
\leq (|CA|+|I_n-C|)x
<\alpha x,
$$
and from this 
\begin{align}\label{ineqPfPropBoundSuffEq2}
|CA|x<\alpha(I_n-|I_n-C|) x.
\end{align}
By the Neumann series theorem and \nref{ineqPfPropBoundSuffEq}, we get
$
(I_n-|I_n-C|)^{-1}=\sum_{k=0}^{\infty}|I_n-C|^k\geq0.
$ 
Hence premultiplying inequality \nref{ineqPfPropBoundSuffEq2} by $(I_n-|I_n-C|)^{-1}$ we obtain
$$
(I_n-|I_n-C|)^{-1}|CA|x<\alpha x.
$$
Now, from 
$$
A=C^{-1}CA=(I_n-(I_n-C))^{-1}CA=\sum_{k=0}^{\infty}(I_n-C)^k CA
$$
we derive
$$
|A|
\leq \sum_{k=0}^{\infty}|I_n-C|^k|CA|
=(I_n-|I_n-C|)^{-1}|CA|.
$$
Putting all together, we obtain
$$
|A|x \leq (I_n-|I_n-C|)^{-1}|CA|x < \alpha x,
$$
Therefore $\rho(|A|)<\alpha$, from which the statement follows due to continuity of the spectral radius \cite{HorJoh1985,Mey2000}.
\end{proof}

\paragraph{Isolated solutions} 
We say that a solution $x^*\in\Ss$ is \emph{isolated} if $N(x^*)\cap\Ss=\{x^*\}$ for some neighborhood $N(x^*)$ of $x^*$.

\begin{proposition}
Let $x^*\in\Ss$ and denote $s^*\coloneqq\sgn(x^*)$. Then $x^*$ is an isolated solution if and only if $A+D_s$ is nonsingular for each $s\in\{\pm1\}^n$ such that $D_s s^*=s^*$.
\end{proposition}

\begin{proof}
The condition says that we inspect all orthants, where $x^*$ lies (even on their borders). The nonsingularity assumption then implies that $x^*$ is a unique solution lying in the interior of a union of certain orthants. Conversely, if $x^*$ is isolated, then it is a unique solution of each processed orthant. Thus $A+D_s$ is nonsingular since otherwise there are infinitely many solutions in the orthant given by~$s$.
\end{proof}

By the proposition, one has to process $2^{k}$ orthants, where $k$ is the cardinality of $\{i\mmid x^*_i=0\}$. 
As we will see later in Proposition~\ref{propBoundNp}, checking if $x^*\coloneqq0$ is an isolated solution is a co-NP-hard problem on a class of AVEs with $b=0$ and $A$ having rank one.

\section{Boundedness}

Obviously, regularity of $[A\pm I_n]$ implies boundedness of the solution set $\Ss(b)$ for an arbitrary right-hand side vector $b\in\R^n$. The converse implication is not true in general; simply consider $A=0$. 
Below, we present a complete characterization of boundedness. By convention, the empty set is bounded.

\begin{proposition}
The set $\Ss(b)$ is bounded for each $b\in\R^n$ if and only if $Ax+|x|=0$ has only the trivial solution $x=0$.
\end{proposition}

\begin{proof}
By Proposition~\ref{propDecomp}, $\Ss(b)$ is bounded if and only if the set described by 
$$
(A+D_s)x=b,\ \ D_sx\geq0
$$
is bounded for each $s\in\{\pm1\}^n$. This set is bounded if and only if the recession cone characterized by
$$
(A+D_s)x=0,\ \ D_sx\geq0
$$
has only the trivial solution (cf.\ \cite{Schr1998}), from which the statement follows.
\end{proof}

Deciding on boundedness is intractable even for rank-one matrices.

\begin{proposition}\label{propBoundNp}
Checking whether $Ax+|x|=0$ has a non-trivial solution is NP-hard on a class of problems with $A$ having natural entries and rank one.
\end{proposition}

\begin{proof}
Consider the matrix in the form $A=ev^T$ for some $v\in\N^n$. The AVE $Ax+|x|=0$ takes the form of $(v^Tx)e+|x|=0$. It has a non-trivial solution if and only if there is a solution such that $|x|=e$, that is, $x=s$ for some $s\in\{\pm1\}^n$. Equivalently, $v^Tx=v^Ts=-1$. Eventually, we again utilize the variant of the NP-hard subset sum problem, in which natural numbers $v_1,\dots,v_n$ are given, and one asks to split them into two subsets the sums of which differ by one.
\end{proof}

Naturally, finiteness of $\Ss(b)$ implies boundedness of $\Ss(b)$. The converse implication does not hold in general. This is illustrated in Figure~\ref{figSolSet1}, where $\Ss(b)$ is bounded but infinite; it can be observed that $\Ss(b)$ is bounded for every $b\in\R^2$ there. 

Thus, for checking boundedness of $\Ss(b)$, we can apply the conditions for finiteness. 
So far, there is no sufficient condition known that would imply boundedness but not necessarily finiteness.

\section{Convexity and connectedness}

\paragraph{Convexity} 
When $[A\pm I_n]$ is regular, the solution set $\Ss(b)$ is a singleton and hence convex for every $b\in\R^n$. The converse implication is not valid: for $A=\left(\begin{smallmatrix}1&2\\2&1\end{smallmatrix}\right)$, the interval matrix $[A\pm I_n]$ is not regular, but $\Ss(b)$ is convex for every $b\in\R^n$. 

Below, we provide a complete characterization of convexity. 
Again, by convention, the empty set is considered as convex.

\begin{proposition}\label{propConvOrth}
The set $\Ss$ is convex if and only if it is located in one orthant only, i.e., there is $s\in\{\pm1\}^n$ such that $D_sx\geq0$ for each $x\in\Ss$.
\end{proposition}

\begin{proof}
\quo{If.} 
Obvious in view of Proposition~\ref{propDecomp}.

\quo{Only if.} 
Let $x^1,x^2\in\Ss$, $\lambda^1,\lambda^2>0$, $\lambda^1+\lambda^2=1$. By definition and convexity of $\Ss$ we have
\begin{align*}
Ax^1+|x^1| &= b,\\
Ax^2+|x^2| &= b,\\
A(\lambda^1x^1+\lambda^2x^2)+|\lambda^1x^1+\lambda^2x^2| &= b,
\end{align*}
from which
\begin{align*}
\lambda^1|x^1|+\lambda^2|x^2|=|\lambda^1x^1+\lambda^2x^2|.
\end{align*}
Hence for every $i=1,\dots,n$ we have
\begin{align*}
\lambda^1|x^1_i|+\lambda^2|x^2_i|=|\lambda^1x^1_i+\lambda^2x^2_i|.
\end{align*}
However, this is possible if and only if $x^1_ix^2_i\geq0$.
\end{proof}

In the expression \nref{condConvMatA} below, the condition $D_{x^1}x^2\geq0$ equivalently reads $x^1_ix^2_i\geq0$ for every $i=1,\dots,n$.

\begin{proposition}
The set $\Ss(b)$ is convex for each $b\in\R^n$ if and only if every $x^1,x^2\in\R^n$ satisfy
\begin{align}\label{condConvMatA}
A(x^1-x^2)=|x^2|-|x^1|\ \Rightarrow\ D_{x^1}x^2\geq0.
\end{align}
\end{proposition}

\begin{proof}
\quo{If.} 
Let $b\in\R^n$ and $x^1,x^2\in\Ss(b)$. Then 
\begin{align*}
Ax^1+|x^1| &= b,\\
Ax^2+|x^2| &= b,
\end{align*}
whence $A(x^1-x^2)=|x^2|-|x^1|$. By \nref{condConvMatA} we have $D_{x^1}x^2\geq0$. So all solutions lie in one orthant and $\Ss(b)$ is convex by Proposition~\ref{propConvOrth}.

\quo{Only if.} 
Suppose to the contrary that 
$A(x^1-x^2)=|x^2|-|x^1|$ for some $x^1,x^2\in\R^n$ and $x^1_ix^2_i<0$ for some $i\in\seznam{n}$. Define $b^*\coloneqq Ax^1+|x^1|=Ax^2+|x^2|$. So we have $x^1,x^2\in\Ss(b^*)$. Since $x^1,x^2$ lie in different orthants, $\Ss(b^*)$ is not convex; a contradiction.
\end{proof}

Notice that condition \nref{condConvMatA} can be checked by a decomposition of space $\R^{2n}$ into orthants since then the condition becomes linear. There is little hope for a simpler condition as the convexity is hard to verify.

\begin{proposition}
Checking convexity of $\Ss$ is co-NP-hard an a class of problems with $b=0$ and $A$ having rank one and natural entries.
\end{proposition}

\begin{proof}
In the proof of Proposition~\ref{propBoundNp}, we proved NP-hardness of checking whether 
the AVE $Ax+|x|=0$ has a non-trivial solution, where $A=ev^T$ has size $n\times n$. Consider now the AVE in $n+1$ variables
\begin{align}\label{eqPfPropConvNp}
(ev^T)x+|x|=0,\ v^Tx+|y|=0.
\end{align}
If the AVE $Ax+|x|=0$ has only the trivial solution, then \nref{eqPfPropConvNp} has only the trivial solution, too. If the AVE $Ax+|x|=0$ has a non-trivial solution $x^*$, then by the proof of Proposition~\ref{propBoundNp} we can assume that $v^Tx^*=-1$. Now, the equation $v^Tx+|y|=0$ reads $|y|=1$. Therefore, \nref{eqPfPropConvNp} has the solutions $(x^*,1)$ and $(x^*,-1)$, but none of their strict convex combinations.
\end{proof}

\paragraph{Connectedness}
The solution set of AVE is connected if there exists a unique solution. Another simple condition is given below. It seems to be a hard problem to characterize connectedness in general.

\begin{proposition}
If $b=0$, then the solution set of AVE is connected.
\end{proposition}

\begin{proof}
In each orthant, the corresponding solution set is connected and contains the origin, via which is the overall solution set connected.
\end{proof}

\section{Optimization reformulations}

AVE can be formulated as an optimization problem~\cite{ManMey2006,Man2007b,Man2007,Man2015,ZamHla2021a}. Each of the problems below has the optimal value $0$ if and only if AVE is solvable. Moreover, the solution of AVE then can be easily deduced then. 
\begin{align}
\label{minSum}
&\min\ e^T(b-Ax-|x|) \st (A+I_n)x\leq b,\ (A-I_n)x\leq b,\\
\label{minAbsSq}
&\min\ (b-Ax-|x|)^T(b-Ax-|x|) \st (A+I_n)x\leq b,\ (A-I_n)x\leq b,\\
\label{minPM}
&\min\ (b-Ax-x)^T(b-Ax+x) \st (A+I_n)x\leq b,\ (A-I_n)x\leq b,\\
&\min\ x^Ty \st (A+I_n)x-(A-I_n)y= b,\ x,y\geq0.
\label{minBilin}
\end{align}
Problem \nref{minSum} is a concave optimization problem with a piecewise linear objective function. The objective function of problem \nref{minAbsSq} is piecewise convex quadratic, but it is not convex on the whole space in general (consider, e.g., the example $A=0$ and $b=1$). Problems \nref{minPM} and \nref{minBilin} are quadratic problems. The objective function of \nref{minPM} reads
$$
(b-Ax-x)^T(b-Ax+x)=x^T(A^TA-I_n)x-2b^TAx+b^Tb,
$$
so it is convex if and only if $A^TA-I_n$ is positive semidefinite. That is the reason why condition \nref{condRegSuff2} implies an efficiently solvable AVE.

In formulations \nref{minSum} to \nref{minPM}, the feasibility system
\begin{align}\label{minFeasMna}
(A+I_n)x\leq b,\ (A-I_n)x\leq b
\end{align}
appears. It plays an important role in the analysis and development of properties of AVE. Notice that it equivalently draws
\begin{align*}
Ax+|x|\leq b.
\end{align*}

\begin{proposition}
For any formulation \nref{minSum} to \nref{minPM} we have that the optimization problem is feasible for each $b\in\R^n$ of and only if the system $Ax+|x|<0$ is solvable.
\end{proposition}

\begin{proof}
\quo{If.} 
Let $b\in\R^n$ be arbitrary and let $x\in\R^n$ be such that $Ax+|x|<0$. Define $x^{\alpha}\coloneqq \alpha x$, where $\alpha>0$ is sufficiently large; one can take any $\alpha\geq\max_i\frac{b_i}{(Ax+|x|)_i}$. Then $Ax^{\alpha}+|x^{\alpha}|\leq b$. 

\quo{Only if.} 
Take $b\coloneqq-e$ and let $x$ be a feasible solution to $(A+I_n)x\leq b$, $(A-I_n)x\leq b$. Then $Ax+|x|\leq b=-e<0$.
\end{proof}

Notice that solvability of the system $Ax+|x|<0$ can be easily checked by linear programming since it is equivalent to solvability of $(A+I_n)x\leq -e$, $(A-I_n)x\leq -e$.

Recall that $z\in\mna{M}$ is the greatest element of a set $\mna{M}\subset\R^n$ if $z\geq y$ for every $y\in\mna{M}$. Further recall that $A$ is a Z-matrix if $a_{ij}\leq0$ for any $i\not=j$.

\begin{proposition}\label{propZmatGrEl}
Let $A$ be a Z-matrix and suppose that the feasible set to \nref{minFeasMna} is nonempty and bounded above. Then \nref{minFeasMna} possesses the greatest element, which solves the AVE and it is the unique solution of the linear program
\begin{align*}
\max\ p^Tx \st (A+I_n)x\leq b,\ (A-I_n)x\leq b
\end{align*}
for any $p>0$.
\end{proposition}

\begin{proof}
First we show that if both $x,y$ solve \nref{minFeasMna}, then $z\coloneqq\max(x,y)$ solves \nref{minFeasMna}. Let $i\in\seznam{n}$ be arbitrary and suppose without loss of generality that $z_i=x_i$. Then
\begin{align*}
(Az)_i+|z_i|
=a_{ii}x_i+\sum_{j\not=i}a_{ij}z_j+|x_i|
\leq a_{ii}x_i+\sum_{j\not=i}a_{ij}x_j+|x_i|
=(Ax)_i+|x_i|
\leq b_i.
\end{align*}
According to the properties of Z-matrices and the greatest/least element theory \cite{CotPan2009} 
we get that there exists the greatest element $z^*$ of \nref{minFeasMna} and it uniquely solves the corresponding linear program. 

It remains to show that $z^*$ solves the AVE. Suppose to the contrary that there is $i$ such that $(Az^*)_i+|z^*_i|<b_i$. Define $z^0\coloneqq z^*+\eps e_i$, where $\eps>0$ is sufficiently small. Then $z^0$ also satisfies \nref{minFeasMna}, which is in contradiction with the fact that $z^*$ is the greatest element of~\nref{minFeasMna}.
\end{proof}

We now show that the assumptions of Proposition~\ref{propZmatGrEl} are satisfied only when $A+I_n$ is an M-matrix.

\begin{proposition}\label{propZmatMmat}
The following statements are equivalent:
\begin{enumerate}[(1)]
\item\label{propZmatMmat1}
$A+I_n$ is an M-matrix.
\item\label{propZmatMmat2}
$A$ is a Z-matrix and the feasible set to \nref{minFeasMna} is nonempty and bounded above for every $b\geq0$.
\item\label{propZmatMmat3}
$A$ is a Z-matrix and the feasible set to \nref{minFeasMna} is nonempty and bounded above for at least one $b\in\R^n$.
\end{enumerate}
\end{proposition}

\begin{proof}
\quo{$\nref{propZmatMmat1}\Rightarrow\nref{propZmatMmat2}$} 
First, notice that $A$ is a Z-matrix. Second, for each $b\geq0$, the point $x_b\coloneqq(A+I_n)^{-1}b\geq0$ fulfills~\nref{minFeasMna}. Third, we show that  \nref{minFeasMna} is bounded above. If it is not the case, then there is a nontrivial point $x^*\geq0$ lying in the recession cone $(A+I_n)x\leq 0$, $(A-I_n)x\leq 0$. However, the first inequality implies $x^*\leq (A+I_n)^{-1}0=0$, whence $x^*=0$; a contradiction.

\quo{$\nref{propZmatMmat2}\Rightarrow\nref{propZmatMmat3}$} 
Obvious.

\quo{$\nref{propZmatMmat3}\Rightarrow\nref{propZmatMmat1}$} 
By the assumption, the system
\begin{align*}
(A+I_n)x\leq b,\ (A-I_n)x\leq b,\ e^Tx\geq\alpha
\end{align*}
is infeasible for $\alpha$ large enough. By the Farkas lemma~\cite{Schr1998}, the dual system
\begin{align*}
(A+I_n)^Tu+(A-I_n)^Tv=ew,\ b^Tu+b^Tv<\alpha w,\ u,v,w\geq0
\end{align*}
has a solution $u^*,v^*,w^*$. Since \nref{minFeasMna} is feasible, $w^*\not=0$. Thus we can assume that $w^*=1$, which reduces $u^*,v^*$ to be a solution to
\begin{align}\label{sysPfProofZmatMmat}
(A+I_n)^Tu+(A-I_n)^Tv=e,\ u,v\geq0.
\end{align}
Suppose that $u^*,v^*$ is a basic solution to \nref{sysPfProofZmatMmat}. A basic solution exists since the equations are linearly independent. To see it, suppose that multiplying the system by a vector $r\in\R^n$ yields an equation $0=0$. Then we deduce $(A+I_n)r=(A-I_n)r=0$, from which $r=0$.

If $(u^*)^Tv^*=0$, then the equation in \nref{sysPfProofZmatMmat} reads $(A+D_s)^Ty=e$, where $y=u^*+v^*\geq0$ and $s=\sgn(u^*-v^*)$. This means that $A+D_s$ is an M-matrix and hence $A+I_n$ is an M-matrix as well. 

If $(u^*)^Tv^*\not=0$, then there is $k$ such that $u_k^*>0$ and $v_k^*>0$. Since $u^*,v^*$ is a basic solution, at least $n$ elements vanish in total. Thus, there is also some $i$ such that $u_i^*=v_i^*=0$. The $i$th equation in \nref{sysPfProofZmatMmat} reads
$$
\sum_{j\not=i}a_{ij}u_j^*+\sum_{j\not=i}a_{ij}v_j^*=1.
$$
However, the left-hand side is nonpositive; a contradiction.
\end{proof}

In the following, we use the vector relation $a\gneqq b$ defined as $a\geq b$, $a\not=b$. 
We also employ the Karush--Kuhn--Tucker (KKT) optimality conditions~\cite{BazShe2006}.

\begin{proposition}
Let $A\in\R^{n\times n}$. In problem \nref{minBilin} for any $b\in\R^n$ each KKT point yields a solution to AVE if and only if 
\begin{align}\label{sysKktOptBilin}
|u|\gneqq|A^Tu|
\end{align}
is infeasible.
\end{proposition}

\begin{proof}
The KKT conditions for problem \nref{minBilin} and a feasible point $(x^*,y^*)$ read
\begin{align}\label{sysPfPropKktOptBilin}
y^*+(A+I_n)^Tu=v\geq0,\ 
x^*-(A-I_n)^Tu=w\geq0,\ 
v^Tx^*=w^Ty^*=0.
\end{align}
The point $(x^*,y^*)$ yields a solution of AVE (in particular, it is the point $x^*-y^*$) if and only if for each $i$ we have $x^*_i=0$ or $y^*_i=0$. 

Let $b\in\R^n$ be given and let $(u,v,w)$ satisfy the KKT conditions. For each $i$ we have $v_i=0$ or $x^*_i=0$. The former implies $-((A+I_n)^Tu)_i=y^*_i\geq0$, and the latter implies $-((A-I_n)^Tu)_i=w_i\geq0$. Thus $-u_i\geq (A^Tu)_i$ or $u_i\geq (A^Tu)_i$ for each $i$, whence $|u|\geq A^Tu$. Similarly from $w^Ty^*=0$ we derive $|u|\geq -A^Tu$, implying $|u|\geq |A^Tu|$.

If $(x^*,y^*)$ does not produce a solution of AVE, then there is $i$ such that $x^*_i>0$ and $y^*_i>0$. Thus $v_i=w_i=0$ and from $((A-I_n)^Tu)_i=x^*_i>0$ and $-((A+I_n)^Tu)_i=y^*_i>0$ we derive $|u|_i>|A^Tu|_i$.

Conversely, suppose that there is $u$ solving~\nref{sysKktOptBilin}. Then we easily reconstruct $x^*$, $y^*$, $v$ and $w$ such that they satisfy the KKT conditions \nref{sysPfPropKktOptBilin}.  
We simply put $y^*\coloneqq\max(0,-(A+I_n)^Tu)$, $v\coloneqq y^*+(A+I_n)^Tu\geq0$, $x^*\coloneqq\max(0,(A-I_n)^Tu)$, $w\coloneqq x^*-(A-I_n)^Tu\geq0$. Further, we take $b\coloneqq (A+I_n)x^*-(A-I_n)y^*$ in order that $(x^*,y^*)$ is feasible. The complementarity conditions will be satisfied, too. Suppose to the contrary there is $i$ such that $v_ix_i^*>0$. Then $y^*_i=w_i=0$, $(A+I_n)^Tu)_i>0$ and $(A-I_n)^Tu)_i>0$. Thus $(A^Tu)_i>|u_i|$, which contradicts~\nref{sysKktOptBilin}.

By the assumption there is $i$ such that $|u|_i>|A^Tu|_i$. By a suitable scaling of $u$ we can assume that $-u_i>|A^Tu|_i$. Then $v_i=w_i=0$, $x^*_i>0$ and $y^*_i>0$. Therefore, $(x^*)^Ty^*\not=0$, and $(x^*,y^*)$ yields no solution of AVE.
\end{proof}

Now, we show that checking \nref{sysKktOptBilin} is intractable.

\begin{proposition}
It is NP-hard to check solvability of \nref{sysKktOptBilin} on a set of rational matrices.
\end{proposition}

\begin{proof}
We use a reduction from irregularity of an interval matrix $[A\pm I_n]$. Notice that its irregularity is equivalent to irregularity of $[A\pm (1+\eps)I_n]$, where $\eps>0$ is sufficiently small but the size of $\eps$ is polynomial (the set of nonsingular matrices is open). By \cite{Roh2009}, $[A\pm (1+\eps)I_n]$ is irregular if and only if there is $u\not=0$ such that 
\begin{align}\label{sysKktOptBilinNPhard}
|u|\geq \frac{1}{1+\eps} |A^Tu|.
\end{align}
Now, if \nref{sysKktOptBilinNPhard} holds with at least one strict inequality, then $[A\pm I_n]$ is irregular. Conversely, let $[A\pm I_n]$ be irregular. Then there is $u\not=0$ such that 
$$
|u|\geq|A^Tu|\geq \frac{1}{1+\eps} |A^Tu|.
$$  
In view of $u\not=0$, there is at least one strict inequality.
\end{proof}

The proof of the above observation shows that condition \nref{sysKktOptBilin} is closely related to regularity of the interval matrix $[A\pm I_n]$. However, it is not exactly the same. Regularity of $[A\pm I_n]$ implies infeasibility of \nref{sysKktOptBilin}, but not conversely (as a counterexample take $A=I_n$). As a consequence, any efficiently recognizable subclass of regular interval matrices (as discussed at the beginning of Section~\ref{sSolv}) serves as a sufficient condition to~\nref{sysKktOptBilin}.

\section{Other auxiliary optimization problems}

There are other optimization problems used to solve AVE. An auxiliary linear program has often the form of \cite{ZamHla2021a}
\begin{align}
\label{minLP}
\min\ e^T(b-(A+D)x) \st (A+I_n)x\leq b,\ (A-I_n)x\leq b
\end{align}
for certain $D:|D|=I_n$.

Below, we formulate the results for the domain $D\in[-I_n,I_n]$, but analogously the results hold for the set of $D$ such that $|D|=I_n$. 
We use $\neg i$ to denote the index set $\seznam{n}\setminus\{i\}$.

\begin{proposition}
The following two statements are equivalent:
\begin{enumerate}[(1)]
\item\label{itPropAuxFeasOpt1}
For every $D\in[-I_n,I_n]$ and $b\in\R^n$, feasibility of \nref{minLP} implies that the optimum of \nref{minLP} solves AVE.
\item\label{itPropAuxFeasOpt2}
For every $D\in[-I_n,I_n]$ and $i\in\seznam{n}$ there is a solution to the system
\begin{align}
\label{sysPropAuxFeasOpt}
(Ax+|x|)_{\neg i} \leq 0,\ e^T(A+D)x>0.
\end{align}
\end{enumerate}
\end{proposition}

\begin{proof}
\quo{$\nref{itPropAuxFeasOpt1}\Rightarrow\nref{itPropAuxFeasOpt2}$} 
Suppose that \nref{sysPropAuxFeasOpt} is infeasible for some $D\in[-I_n,I_n]$ and $i\in\seznam{n}$. Define $b\coloneqq e_i$. Then $x^*\coloneqq 0$ is obviously a feasible solution to \nref{minLP}. In view of infeasibility of \nref{sysPropAuxFeasOpt}, $x^*$ is an optimum to \nref{minLP} since there is no improving direction. However, $x^*$ does not solve AVE since $Ax^*+|x^*|=0\not=b$.

\quo{$\nref{itPropAuxFeasOpt2}\Rightarrow\nref{itPropAuxFeasOpt1}$} 
From the assumption and by the Farkas lemma~\cite{Schr1998}, the system
\begin{align*}
((A+I_n)_{\neg i,1:n})^T\tilde{u}+((A-I_n)_{\neg i,1:n})^T\tilde{v}=(A+D)^Te,\ 
\tilde{u},\tilde{v}\geq0
\end{align*}
is infeasible for each $D\in[-I_n,I_n]$ and $i\in\seznam{n}$. That is why for the dual problem to \nref{minLP},
\begin{align*}
\max\ -b^T(u+v) \st (A+I_n)^Tu+(A-I_n)^Tv = (A+D)^Te,\ u,v\geq0,
\end{align*}
we can claim that each feasible solution $u,v$ satisfies $u+v>0$ (that is, $u_i+v_i>0$ for each~$i$). The primal problem \nref{minLP} has an optimal solution since it is feasible and the objective function is linear and bounded below. For a primal optimum $x^*$ and a dual optimum $(u^*,v^*)$ the optimality conditions are true,
$$
((A+I_n)x^*-b)^Tu^*=((A-I_n)x^*-b)^Tv^*=0.
$$
Since $u^*+v^*>0$, we have for each $i$ that $((A+I_n)x^*)_i=b_i$ or $((A-I_n)x^*)_i=b_i$, giving raise to $(Ax^*+|x^*|)_i=b_i$. Therefore $Ax^*+|x^*|=b$.
\end{proof}

Despite the fact that condition \nref{sysPropAuxFeasOpt} can be checked by linear programming, it is desirable to have some matrix classes that satisfy it automatically.

Let $\umace{A}$ and $\omace{A}$ be symmetric. We say that an interval matrix $[\umace{A},\omace{A}]$ is \emph{positive definite} if every symmetric $A\in[\umace{A},\omace{A}]$ is positive definite. Checking positive definiteness of an interval matrix is co-NP-hard \cite{Roh1994}, but there are known various sufficient conditions \cite{Hla2018a,Hla2020a,Roh1994b}.

\begin{proposition}\label{propAuxFeasOptPD}
Let $A\in\R^{n\times n}$ be symmetric. Condition \nref{sysPropAuxFeasOpt} holds true provided $[A\pm ee^T]$ is positive definite.
\end{proposition}

\begin{proof}
Let $D\in[-I_n,I_n]$ and suppose without loss of generality that $i=n$. 
For every $D'\in[-I_n,I_n]$, we find a solution to the system
\begin{align*}
((A+D')x)_{\neg n} \leq 0,\ e^T(A+D)x>0;
\end{align*}
by \cite{Hla2013b,RohKre1994}, it is equivalent to having a solution to \nref{sysPropAuxFeasOpt}.

Let the matrices have the structure
\begin{align*}
A=\begin{pmatrix}\tilde{A} & a\\a^T & c\end{pmatrix},\ \ 
D'=\begin{pmatrix}\tilde{D}' & 0\\0^T & \gamma\end{pmatrix},\ \ 
D=\begin{pmatrix}\tilde{D} & 0\\0^T & -1\end{pmatrix}. 
\end{align*}
Notice that above we can assume that $D_{nn}=-1$ since we will construct a solution $x$ such that $x_n=1$. Indeed, we seek for a solution in the form $x=(\tilde{x}^T,1)^T$. To satisfy the inequalities $((A+D')x)_{\neg n} \leq 0$ as equations, we put $\tilde{x}\coloneqq-(\tilde{A}+\tilde{D}')^{-1}a$. Thus it remains to show that $x$ satisfies the remaining inequality $e^T(A+D)x>0$. Substituting for $x$, we get
\begin{align*}
e^T(A+D)x
&=(e^T,1) 
 \begin{pmatrix}(\tilde{A}+\tilde{D}')+(\tilde{D}-\tilde{D}') & a\\a^T & c-1\end{pmatrix}
 \begin{pmatrix}-(\tilde{A}+\tilde{D}')^{-1}a \\ 1\end{pmatrix} \\
&= -e^T(\tilde{D}-\tilde{D}') (\tilde{A}+\tilde{D}')^{-1}a
   -a^T(\tilde{A}+\tilde{D}')^{-1}a + c-1
\end{align*}
Denote $d\coloneqq \frac{1}{2}(\tilde{D}-\tilde{D}')e$. Then the above expression reads
\begin{align*}
e^T(A+D)x
& =-(a+d)^T (\tilde{A}+\tilde{D}')^{-1}(a+d) + d^T(\tilde{A}+\tilde{D}')^{-1}d + c-1\\
&\geq -(a+d)^T (\tilde{A}+\tilde{D}')^{-1}(a+d) + c-1\\
&>0.
\end{align*}
The last inequality follows from positive definiteness of $[A\pm ee^T]$ since the value $-(a+d)^T (\tilde{A}+\tilde{D}')^{-1}(a+d) + c-1$ is in fact the Schur complement of 
\begin{align*}
\begin{pmatrix}\tilde{A}+\tilde{D}' & a+d\\(a+d)^T & c-1\end{pmatrix}
\in[A\pm ee^T].
\end{align*}
\end{proof}

Notice that positive definiteness of $[A\pm ee^T]$ already implies its regularity and thus unique solvability, however, the solution need not be easy to compute in general. In our case, the linear program \nref{minLP} does the job.

It is also worth mentioning that the class of matrices satisfying \nref{sysPropAuxFeasOpt} is closed under multiplication of any column by $-1$; one just uses the substitution $x_i\equiv-x_i$. As a consequence, with any $A$ the class also contains $-A$. Thus  \nref{sysPropAuxFeasOpt} holds true when $[A\pm ee^T]$ is negative definite, too.

An interval matrix $[\umace{A},\omace{A}]$ is \emph{an M-matrix} if every $A\in[\umace{A},\omace{A}]$ is an M-matrix, that is, the off-diagonal entries are nonpositive and $A^{-1}\geq0$. Checking M-matrix property is easy since $[\umace{A},\omace{A}]$ is an M-matrix if and only if $\umace{A}$ is an M-matrix and $\omace{A}_{ij}\leq0$ for $i\not=j$; see \cite{BarNud1974}. 

\begin{proposition}
Let $A\in\R^{n\times n}$ be symmetric. Condition \nref{sysPropAuxFeasOpt} holds true provided 
$[A\pm(2ee^T-I_n)]$ is an M-matrix.
\end{proposition}

\begin{proof}
We proceed similarly as in the proof of Proposition~\ref{propAuxFeasOptPD} and we use the same notation. However, matrix $A$ is now structured as
\begin{align*}
A=\begin{pmatrix}\tilde{A} & a\\g^T & c\end{pmatrix}
\end{align*}
and the expression $e^T(A+D)x$ is estimated from below as follows
\begin{align*}
e^T(A+D)x
 = -(g+2d)^T (\tilde{A}+\tilde{D}')^{-1}a  + c-1 
 > 0.
\end{align*}
The last inequality follows from the M-matrix property since the value $-(g+2d)^T (\tilde{A}+\tilde{D}')^{-1}a  + c-1$ is the Schur complement of 
\begin{align*}
\begin{pmatrix}\tilde{A}+\tilde{D}' & a\\(g+2d)^T & c-1\end{pmatrix}
\in[A\pm(2ee^T-I_n)],
\end{align*}
so it is also an M-matrix \cite{HorJoh1991}.
\end{proof}

\section{Conclusion}

As Cottle~\cite{Cot2010} writes, \quo{there are more than 50 matrix classes discussed in the literature of the Linear Complementarity Problem}. In contrast, the matrix classes associated with the absolute value equations have not been thoroughly investigated so far. The purpose of this paper was to change it and to address this issue. We introduced (without explicitly mentioning it) the matrix classes corresponding to convexity, boundedness, connectedness, finiteness, nonnegativity and other properties of the solution set. Other matrix classes are related to optimality conditions of optimization reformulations.

We were able to characterize some of the matrix classes and decide on the computational complexity of the corresponding recognition problem. There are, however, several questions left as open problems.


\bibliographystyle{abbrv}
\bibliography{mat_class_ave}

\end{document}